\documentclass{conm-p-l}

\usepackage[titlenumbered,algoruled]{algorithm2e}
\usepackage{}

\newtheorem{theorem}{Theorem}[section]

\newtheorem{corollary}[theorem]{Corollary}

\theoremstyle{definition}
\newtheorem{definition}[theorem]{Definition}
\newtheorem{example}[theorem]{Example}

\theoremstyle{remark}
\newtheorem{remark}[theorem]{Remark}

\numberwithin{equation}{section}

\newcommand{\Q}{\mathbb{Q}}
\newcommand{\pj}{\mathbb{P}}
\newcommand{\fsl}{\mathfrak{sl}}
\newcommand{\fgl}{\mathfrak{gl}}
\DeclareMathOperator{\PGL}{PGL}
\DeclareMathOperator{\Resultant}{Resultant}
\DeclareMathOperator{\Sym}{Sym}

\begin{document}

\title{Effective Radical Parametrization of Trigonal Curves}


%    author one information
\author{Josef Schicho}
\address{Johann Radon Institute for Computational and Applied Mathematics (RICAM), Austrian Academy of Sciences, Altenbergerstrasse 69, A-4040 Linz, Austria}
\curraddr{}
\email{josef.schicho@oeaw.ac.at, david.sevilla@oeaw.ac.at}
\thanks{Partially supported by the Austrian FWF project P 22766-N18 "Radical parametrizations of algebraic curves"}

%    author two information
\author{David Sevilla}
%\address{Johann Radon Institute for Computational and Applied Mathematics (RICAM), Austrian Academy of Sciences, Altenbergerstrasse 69, A-4040 Linz, Austria}
%\curraddr{}
%\email{david.sevilla@oeaw.ac.at}
%\thanks{Partially supported by the Austrian FWF project P 22766-N18 "Radical parametrizations of algebraic curves"}

%\subjclass[2000]{Primary }
%    The 2010 edition of the Mathematics Subject Classification is
%    now available.  If you are citing a classification from the
%    new scheme, use the following input coding instead.

\date{}

\begin{abstract}
Let $C$ be a non-hyperelliptic algebraic curve. It is known that its canonical image is the intersection of the quadrics that contain it, except when $C$ is trigonal (that is, it has a linear system of degree 3 and dimension 1) or isomorphic to a plane quintic (genus 6). In this context, we present a method to decide whether a given algebraic curve is trigonal, and in the affirmative case to compute a map from $C$ to the projective line whose fibers cut out the linear system.
\end{abstract}

\maketitle

\section{Introduction}

In the context of symbolic computation for algebraic geometry, an unsolved problem (at least from a computational perspective) is the parametrization of algebraic curves by radicals. Allowing radicals rather than just rational functions greatly enlarges the class of parametrizable functions. For example, one class of curves which are clearly parametrizable by radicals is that of hyperelliptic curves. Every such curve can be written as $y^2 = P(x)$ for some polynomial $P(x)$, and we can quickly write the parametrization $x=t,y=\sqrt{P(t)}$.

This can be taken further: the roots of univariate polynomials of degree $\leq4$ can be written in terms of radicals. Therefore, curves which can be expressed as $f(x,y)=0$ where one of the variables occurs with degree $\leq4$ can also be parametrized by radicals. The minimum degree which can be obtained by is called the \emph{gonality of the curve}; hyperelliptic curves are precisely those of gonality two. It is thus interesting to characterize the curves of gonality three (or \emph{trigonal}) and four, and further to produce algorithms that detect this situation and can even compute a radical parametrization. The description of such an algorithm for trigonal curves is the purpose of this article.

In this article, the coefficient field always has characteristic zero, and it will generally be assumed to be algebraically closed although we will point out the necessary modifications for the non-algebraically-closed case when they arise. Our algorithm is based on the \emph{Lie algebra method} introduced in \cite{deGraafHarrisonPilnikovaSchicho2006a} (see also \cite{deGraafPilnikovaSchicho2009a}). We use Lie algebra computations (which mostly amount to linear algebra) to decide if a certain algebraic variety associated to the input curve is a rational normal scroll, which is the case precisely when the curve is trigonal. Further, we can compute an isomorphism between that variety and the scroll when it exists. Algorithm \ref{ALG: sketch detect trigonality} sketches the classification part of the algorithm (that is, the detection of trigonality as opposed to the calculation of a $3:1$ map). It is based on Theorem \ref{TH: classification of the quadrics}.

\begin{algorithm}
  %\dontprintsemicolon
  \caption{Sketch of algorithm to detect trigonality}\label{ALG: sketch detect trigonality}
  \BlankLine
  \KwIn{a non-hyperelliptic curve $C$ of genus $g\geq3$}
  \KwOut{\texttt{true} if $C$ is trigonal, \texttt{false} otherwise}
  \BlankLine
  Compute the canonical map $\varphi\colon C\rightarrow\pj^{g-1}$ and its image $\varphi(C)$\;
  Compute the intersection $D$ of all the quadrics that contain $\varphi(C)$\;
  \eIf{$D=C$}{
    \Return false\;
  }{
    Determine which type of surface is $D$\;
    \uIf{$D=\pj^2$}{
      \Return true \hspace{5ex}\tcp*[h]{$g=3$}\;
    }
    \uElseIf{$D$ is a rational normal scroll}{
      \Return true\;
    }
    \Else{\Return false \hspace{5ex}\tcp*[h]{Veronese} \;}
  }
\end{algorithm}

The article is structured as follows. Section \ref{SEC: Trigonality background} recalls the classical theoretical background on trigonal curves and rational scrolls needed. Section \ref{SEC: Lie algebras} is a quick survey of the relevant concepts of Lie algebras and their representations. Section \ref{SEC: Lie algebra method} describes the method proper. Our computational experiences with it are reported in Section \ref{SEC: Computational experiences}.

\section{Classical results on trigonality}\label{SEC: Trigonality background}

Let $C$ be an algebraic curve of genus $g\geq4$ and assume that $C$ is not hyperelliptic, so that it is isomorphic to its image by the canonical map $\varphi\colon C\rightarrow\pj^{g-1}$. Enriques proved in \cite{Enriques1919a} that $\varphi(C)$ is the intersection of the quadrics that contain it, except when $C$ is trigonal (that is, it has a $g^1_3$) or isomorphic to a plane quintic ($g=6$); the proof was completed in \cite{Babbage1939a}. In those cases, the corresponding varieties are minimal degree surfaces, see \cite[p. 522 and onwards]{GriffithsHarris1978a}. Although not relevant in our case it is worth mentioning that \cite{Petri1923a} proves that the ideal is always generated by the quadrics and cubics containing the canonical curve.

We exclude from our study the curves with genus lower than 3 since they are hyperelliptic, thus they have a $g^1_2$ which can be made into a $g^1_3$ by adding a base point; the problem is then to find a point in the curve over the field of definition. Also, if the curve is non-hyperelliptic of genus 3, it is isomorphic to its canonical image which is a quartic in $\pj^2$, and the system of lines through any point of the curve cuts out a $g^1_3$.

There exist efficient algorithms for the computation of the canonical map, determination of hyperellipticity, and calculation of the space of forms of a given degree containing a curve, for example in Magma \cite{BosmaCannonPlayoust1997a} and at least partially in Maple.

The following theorem summarizes the classification of canonical curves according to the intersection of the quadric hypersurfaces that contain them.

\begin{theorem}[{\cite[p. 535]{GriffithsHarris1978a}}]\label{TH: classification of the quadrics}
For any canonical curve $C\subset\pj^{g-1}$ over an algebraically closed field, either
\begin{enumerate}
 \item $C$ is entirely cut out by quadric hypersurfaces; or
 \item $C$ is trigonal, in which case the intersection of all quadrics containing $C$ is isomorphic to the rational normal scroll swept out by the trichords of $C$; or
 \item $C$ is isomorphic to a plane quintic, in which case the intersection of the quadrics containing $C$ is isomorphic to the Veronese surface in $\pj^5$, swept out by the conic curves through five coplanar points of $C$.
\end{enumerate}
\end{theorem}

 We recall the definition of \emph{rational normal scroll} (from this point, simply scroll): given two nonnegative integers $m\geq n$ with $m+n\geq2$, the scroll $S_{m,n}$ is the Zariski closure of the image of $(s,t)\mapsto(1:s:s^2:\ldots:s^m:t:st:s^2t:\ldots:s^nt)\subset\pj^{m+n+1}$. It is defined by equations of degree two involving four terms each. It is a ruled surface, its pencil of lines being given by the fibers $s=constant$. This ruling is unique except when $m=n=1$ in which case there are two rulings. Any map $S_{m,n}\rightarrow\pj^1$ whose fibers are lines is called a \emph{structure map}.

If we work over a non-algebraically closed field $k$, one may get surfaces $S$ which are isomorphic to scrolls only over $\overline k$ (called twists). The structure map $S\rightarrow\pj^1$ over $\overline k$ is given by a divisor class, which may be not defined over $k$, or may be defined over $k$ but have no divisors in it over $k$. The first case can only occur for genus 4, see Section \ref{SEC: Case m=n} for the details.

In the second case (the divisor class is defined over $k$ but has no elements over it), one can define a map $S\rightarrow E$ whose fibers are lines, where $E$ is a conic over $k$ with no points defined over it (see \cite{PerezSendraWinkler2008a}). This is done by taking the structure map $S\rightarrow\pj^1$ and symmetrizing it with its Galois conjugates over $k$. Since $E$ has no points defined over $k$, neither does $S$ or the trigonal curve $C$. However, one can always go to a degree 2 extension of $k$ where $E$ has a point and work on that extension.

Let $M$ be a chosen algebraic variety (a ``model'') and $X\in\pj^N$ be any given variety. In some cases, we can use Lie algebra representations (section \ref{SEC: Lie algebras}) to decide if $X$ is projectively isomorphic to $M$ (we call this \emph{recognition} of $M$) and furthermore to compute a projective isomorphism between them in the affirmative case (we call this \emph{constructive recognition}). We will use the same terms for Lie algebras.

Strictly speaking, we are not interested in constructive recognition of scrolls, since we only need the structure map whose fibres will cut out the trigonal linear system. However, it is possible to use the method described below to construct isomorphisms to the models of the scrolls, as we will comment in each case.

\section{Lie algebras}\label{SEC: Lie algebras}

The Lie algebra of a projective variety is an algebraic invariant which is relatively easy to calculate when the variety is generated by quadrics (it is often cheaper than a Gr\"obner basis of the defining ideal, if only generators are given). We offer here a quick summary of relevant properties of Lie algebras in general, see \cite{deGraaf2000a,FultonHarris1991a} for a general overview. Most definitions and basic results can be found in the aforementioned references, we limit ourselves to what is relevant for our purposes.

\begin{definition}
Let $X\subset\pj^N$ be an embedded projective variety. The group of automorphisms of $\pj^N$ is $\PGL_{N+1}$, the group of all invertible matrices of size $N+1$ modulo scalar matrices. Let $\PGL_{N+1}(X)$ be the subgroup of all projective transformations that map $X$ to itself (this is always an algebraic group). The Lie algebra $L(X)$ of $X$ is defined as the tangent space of $\PGL_{N+1}(X)$ at the identity, together with its natural Lie product.
\end{definition}

\begin{example}\label{EX: L(P^N)=sl_{N+1}, sl_2 and basis} $ $
 \begin{enumerate}
  \item The Lie algebra of $\pj^N$ is the tangent space of $\PGL_{N+1}$ at the identity matrix; this is denoted $\fsl_{N+1}$ and its elements are trace zero matrices.
  \item In particular, $L(\pj^1)=\fsl_2$, whose elements are $2\times2$ trace zero matrices. It has dimension 3 and it has a basis
  \[
   h:=\begin{pmatrix} 1 & 0 \\ 0 & -1 \end{pmatrix} \, , \quad x:=\begin{pmatrix} 0 & 1 \\ 0 & 0 \end{pmatrix} \, , \quad y:=\begin{pmatrix} 0 & 0 \\ 1 & 0 \end{pmatrix}
  \]
  This particular basis is an instance of the so-called \emph{Chevalley basis} or \emph{canonical basis}, see \cite[Section 5.11]{deGraaf2000a}.
 \end{enumerate}
\end{example}

\begin{remark}
Note that the Lie algebra of any $X\subset\pj^N$ is a subalgebra of $\fsl_{N+1}$, since $\PGL_{N+1}(X)$ is a subgroup of $\PGL_{N+1}$ so the same relation holds for their tangent spaces.
\end{remark}

For varieties of general type (in particular curves of genus at least 2), $\PGL_{N+1}(X)$ is finite and therefore the Lie algebra is zero. On the other hand, the Veronese surface and the rational scrolls have Lie algebras of positive dimension. This allows us to reduce the recognition problem for these surfaces to  Lie algebra computations (see Section \ref{SEC: Lie algebra method}). The next theorem provides a fast way to compute the Lie algebra of the varieties we are interested in.

\begin{theorem}
Let $X\subset\pj^N$ such that $I(X)$ is generated by quadrics, and $G$ be a set of quadratic generators. Then
\[
 L(X)=\left\{M\in\fsl_{N+1}:\ \forall f\in G,\ \frac{d}{dt}f(I_N+tM)\bigg|_{t=0}\in I\right\}
\]
where by $\frac{d}{dt}f(I_N+tM)$ we mean the function $\frac{d}{dt}f$ applied to the image of the vector of variables by the transformation $I_N+tM$.
\end{theorem}

\begin{proof}
See \cite[Theorem 5]{deGraafPilnikovaSchicho2009a}.
\end{proof}

\begin{remark}
The previous theorem may be true without the hypothesis on the generators, but we do not know a proof; however this restricted version is all that we need for the upcoming discussion.
\end{remark}

\subsection{Representations of Lie algebras}\label{SEC: Reps of Lie algebras}

As in the case of groups, one can understand a lot of things about Lie algebras by thinking of them as spaces of matrices. This is the meaning of the concept of representation that we introduce now.

\begin{definition}
A \emph{representation} of a Lie algebra $L$ is a Lie algebra homomorphism $L\rightarrow\fgl(V)$ for some vector space $V$. This is equivalent to a bilinear action $L\times V\rightarrow V$ which turns $V$ into a \emph{Lie module over $L$}. The \emph{dimension} of a representation is the dimension of $V$.
\end{definition}

\begin{definition}$ $
\begin{enumerate}
 \item A linear subspace $W$ of $V$ is a \emph{submodule} iff $L\cdot W\subseteq W$, that is, if the action can be restricted to $W$.
 \item A module is \emph{irreducible} iff it only has trivial submodules (the trivial and total subspaces).
\end{enumerate}
\end{definition}

An important class of Lie algebras, which are basic building blocks in the classification theory of Lie algebras and the study of their representations, is that of \emph{semisimple} ones. The following convenient result holds:

\begin{theorem}[Weil's theorem]
If $L$ is semisimple, every finite-dimensional module over $L$ is a direct sum of irreducible modules.
\end{theorem}

\begin{proof}
See \cite[Section 4.4]{deGraaf2000a}.
\end{proof}

\subsection{The Lie algebra $\fsl_2$}\label{SEC: sl2}

A particularly important Lie algebra is $\fsl_2$ (Example \ref{EX: L(P^N)=sl_{N+1}, sl_2 and basis}). Since it will feature often in this article, we highlight here some of the properties that we will use.

Computationally speaking, it is easy to recognize $\fsl_2$. First we need a definition.

\begin{definition}
Let $L$ be a Lie algebra. For each $x\in L$, define a Lie endomorphism $ad_x\colon L\rightarrow L$ as $ad_x(y)=[x,y]$. Then the \emph{Killing form on $L$} is a bilinear map $B_L:L\times L\rightarrow k$ given by $B_L(x,y)=Trace(ad_X\circ ad_Y)$.
\end{definition}

\begin{theorem}\cite[Proposition 10]{deGraafPilnikovaSchicho2009a}
Let $L$ be a semisimple Lie algebra of dimension 3. Then $L$ is isomorphic to $\fsl_2$ iff its Killing form is isotropic.
\end{theorem}

In particular, over an algebraically closed field $\fsl_2$ is the only semisimple algebra of dimension 3.

Once we have recognized $L\cong\fsl_2$, the construction of a Lie algebra isomorphism can be done by finding a Chevalley basis of $L$, see the proof of the previous theorem. In general, the procedure amounts to finding a rational point on a conic over the field of definition, which in the case of $\Q$ needs factorization of integers.

The irreducible representations of $\fsl_2$ can be described as follows: for every nonnegative integer $n$, there exists a unique irreducible representation (up to isomorphism) of dimension $n+1$. For $n=0$ this is just the trivial representation. For positive $n$, we describe it as an action of $\fsl_2$ on the $n+1$-dimensional vector space of homogeneous polynomials of degree $n$ in the variables $x,y$:
\[
 P(x,y) \cdot \begin{pmatrix} a & b \\ c & d \end{pmatrix} = P(ax+by,cx+dy)
\]
In terms of modules, they consist of a 1-dimensional module with the zero action, a 2-dimensional module $N$, and the $(n+1)$-dimensional symmetric powers $\Sym^n(N)$ for $n\geq2$.

\begin{remark}\label{REM: Images of h,x,y in Sym^n}
By choosing an adequate basis of eigenvectors of the image of $h$ by the representation, the images of the Chevalley basis can be taken to be
\[
 \setlength{\arraycolsep}{0.01em} h\mapsto
 \begin{pmatrix} n\\&n-2\\&&\ddots\\&&&-n+2\\&&&&-n \end{pmatrix}, \quad
 \setlength{\arraycolsep}{0.3em} x\mapsto
 \begin{pmatrix} 0&1\\&0&2\\&&\ddots&\ddots\\&&&0&n\\&&&&0 \end{pmatrix}, \quad
 y\mapsto
 \begin{pmatrix} 0\\n&0\\&\ddots&\ddots\\&&2&0\\&&&1&0 \end{pmatrix}
\]
See \cite[Section 5.1]{deGraaf2000a} and \cite[Lecture 11]{FultonHarris1991a}.
\end{remark}

\section{The Lie algebra method for trigonal curves}\label{SEC: Lie algebra method}

%Let $M$ be a chosen algebraic variety (a ``model'') and $X\in\pj^N$ be any given variety. In some cases, we can use Lie algebra representations to decide if $X$ is projectively isomorphic to $M$ (we call this \emph{recognition} of $M$) and furthermore to compute a projective isomorphism between them in the affirmative case (we call this \emph{constructive recognition}). We will use the same terms for Lie algebras.

Consider the problem of recognizing $S_{m,n}$. In fact, since we normally have only a variety $X$ but not $m,n$, we want to decide if $X$ is isomorphic to $S_{m,n}$ for some unknown $m,n$. The knowledge of the representations of the Lie algebras of the scrolls, in fact of semisimple parts of them, will allow us to decide the answer and even to compute such an isomorphism.

\begin{definition}
Every finite-dimensional Lie algebra $L$ can be written as a semidirect sum of two parts called a solvable part and a semisimple part. The latter is called a \emph{Levi subalgebra} of $L$, and it is unique up to conjugation, so we will speak of ``the'' Levi subalgebra of $L$ and denote it as $LSA(L)$. For a variety $X$, we will denote $LSA(L(X))$ simply by $LSA(X)$.
\end{definition}

As mentioned before, the Lie algebra of a curve of genus 2 or higher is zero since its automorphism group is finite. The rest of the cases that arise in Theorem \ref{TH: classification of the quadrics} are studied in the next result.

\begin{theorem}
Let $k$ be an algebraically closed field of characteristic zero. As above, let $S_{m,n}$ be the the rational normal scroll with parameters $m,n$, and let $V$ be the image of the Veronese map $\pj^2\rightarrow\pj^5$.
 \begin{enumerate}
  \item $LSA(S_{m,n})\cong\fsl_2$ if $m\neq n$.
  \item $LSA(S_{m,m})\cong\fsl_2+\fsl_2$ (a direct sum of two Lie algebras)
  \item $LSA(V)\cong\fsl_3$.
 \end{enumerate}
\end{theorem}

\begin{proof}
See \cite[Section 3.4]{Oda1988a}. Additionally, for any fixed pair $m,n$ one can easily check the claim, for example in Magma.
  % \begin{enumerate}
  % \item From the parametrization $(s,t)\mapsto(1:s:\ldots:s^n:t:\ldots:s^mt)$ we get a map $S_{m,n}\rightarrow\pj^1$ whose fibres are lines. Since every linear automorphism sends lines to lines, this induces a map $\Aut_l(S_{m,n})\rightarrow\Aut(\pj^1)=\PGL(2)$. One can check that this is surjective and in passing to Lie algebras we obtain $L(S_{m,n})\rightarrow\fsl_2$, which is surjective as well\mnote{why?}. The quotient is solvable, so it has to be the solvable radical\mnote{still unclear} and $\fsl_2$ is the Levi subalgebra.
  % \item 
  % \item 
  % \end{enumerate}
\end{proof}

It is clear now that just by looking at the dimension of the Levi subalgebra we can discard the two cases where the curve is not trigonal. In other words, we can recognize a trigonal curve by the dimension of its Levi subalgebra.

\begin{corollary}
Let $k$ be any field of characteristic zero, let $C$ be a canonical curve and $X$ be the intersection of the quadrics that contain it. Then one of the following occurs:
\begin{itemize}
\item If $\dim LSA(X)=0$ then $X=C$ and $C$ is not trigonal.
\item If $\dim LSA(X)=3$ then $X$ is a twist of $S_{m,n}$ with $m\neq n$ and $C$ is trigonal.
\item If $\dim LSA(X)=6$ then $X$ is a twist of $S_{m,m}$ and $C$ is trigonal.
\item If $\dim LSA(X)=8$ then $X\cong V$ and $C$ is not trigonal.
\end{itemize}
\end{corollary}

\begin{proof}
Pass to $\overline{k}$ and apply the previous theorem.
\end{proof}

If the surface is the Veronese surface, the algorithm terminates and reports that the curve is not trigonal. Nevertheless, for the sake of completeness we must mention that an analysis similar to the scrolls below can be performed and results in the constructive recognition of the Veronese surface; this provides an isomorphism to a plane quintic curve, and knowing a point the pencil of lines through that point will give a $g^1_4$. This only occurs for genus 6.%, and it is in accordance with Brill-Noether theory, which predicts that every curve of genus 6 has a $g^1_4$.

% What remains is constructive recognition: the construction of a surface isomorphism to a suitable scroll which yields the $3:1$ map to $\pj^1$.

% It is clear that an isomorphism between two algebraic varieties induces an isomorphism of their Lie algebras. The converse is not always true, but in our case the classification of minimal degree surfaces (see Section \ref{SEC: Trigonality background}) provides this as well. 

% \begin{theorem}\label{TH: isoms of irred reps of semisimple LAs are almost unique}
% Isomorphisms between irreducible representations of semisimple Lie algebras are unique up to scalar multiplication.\mnote{ref?}
% \end{theorem}

% \begin{theorem}\label{TH: LA isoms correspond to surface isoms}\mnote{Check correctness and wording: cited th. only for irreducible!}
% Given an isomorphism $\lambda$ between the Lie algebras of two surfaces of minimal degree, there exists an isomorphism of the surfaces which induces it.
% \end{theorem}

% \begin{proof}
% Since minimal degree surfaces are classified by their Lie algebras\mnote{check!!}, a Lie algebra isomorphism can only exist if the two surfaces are isomorphic; call that surface isomorphism $\phi$. By Theorem \ref{TH: isoms of irred reps of semisimple LAs are almost unique}, the linear isomorphism of Lie modules induced by $\lambda$ produces the projective isomorphism $\phi$ precisely.
% \end{proof}

\subsection{The case $m\neq n$}\label{SEC: Case m<>n}

Since the Levi subalgebras of these scrolls are always $\fsl_2$, and thanks to the classification above, we have a necessary and sufficient condition for to the recognition problem. Following Section \ref{SEC: sl2} we assume that we have constructed an isomorphism $\sigma\colon\fsl_2\rightarrow LSA(X)$.

% Then we have two representations of $\fsl_2$: the first one, say $\alpha$, is given by the inclusion of $\fsl_2$ into $L(S_{m,n})\subset\fgl_{m+n+2}$, and the other one, say $\omega$, is the composition of $\sigma$ and the inclusion $LSA(X)\subset L(X)\subset\fgl_{N+1}$. Now, if there exists an isomorphism $\rho\colon S_{m,n}\rightarrow X$ then by Remark \ref{REM: X and Smn can live in the same pj^N}\mnote{... which was wrong!} we have $N+1=m+n+2$ and $\rho$ is given by an invertible $(N+1)\times(N+1)$ matrix $Q$. Also, $\rho$ induces $\rho_*\colon L(S_{m,n})\rightarrow L(X)$ given by $\rho_*(l)=QlQ^{-1}$ for this $Q$\mnote{unclear}. But $Q$ also defines an isomorphism of the two modules $V_\alpha,V_\omega$ corresponding to the representations $\alpha,\omega$. %The strategy is to determine if such an isomorphism indeed exists (otherwise the varieties are not isomorphic), and to recover $Q$ from it\mnote{this isomorphism is already $Q$, right?}.
% Once we have $Q$, this is in fact a linear map between the vector spaces in whose projectivizations the varieties live in; by Theorem \ref{TH: LA isoms correspond to surface isoms} we have the sought for surface isomorphism.

The representation of $\fsl_2$ given by the inclusion of $S_{m,n}$, $m\neq n$ into projective space is known.

\begin{theorem}
Consider $S_{m,n}\subset\pj^{m+n+1}$, $m\neq n$. The module $\fsl_2$-module induced on the underlying vector space $V\cong k^{m+n+2}$ decomposes into irreducible modules as $\Sym^m(k^2)\oplus\Sym^n(k^2)$.
\end{theorem}

\begin{proof}
See \cite[Section 3.4]{Oda1988a}. Additionally, for any fixed pair $m,n$ one can easily check the claim, for example in Magma.
\end{proof}

\begin{remark}\label{REM: Images of h,x,y in two-block rep}
Consider the representation of $\fsl_2$ corresponding to the module structure, $Rep\colon\fsl_2\rightarrow\fgl(k^{m+n+2})$. With respect to a suitable basis, the matrices of the elements $h,x,y$ in a Chevalley basis will consist on two blocks of dimensions $m+1$ and $n+1$ having the form given in Remark \ref{REM: Images of h,x,y in Sym^n}. Thus $Rep(h)$ is a diagonal matrix with eigenvalues $m,m-2,\ldots,-m,n,n-2,\ldots,-n$, and
%By Remark \ref{REM: Images of h,x,y in Sym^n} there is a basis of $V_1$ such that the matrices $\alpha(h),\alpha(x),\alpha(y)$ consist of two blocks of the form indicated there. In particular, the vectors in the basis are eigenvectors whose eigenvalues are $m,m-2,\ldots,-m,n,n-2,\ldots,-n$. If the eigenvalues of $\omega(h)$ do not follow that pattern, there is no isomorphism. Otherwise, let $m>n$, so that $m$ is the largest eigenvalue of $\omega(h)$. Let $v_0$ be an eigenvector of eigenvalue $m$. Define $v_i=\frac{1}{m-i+1}\cdot\omega(y)(v_{i-1})$\mnote{why do we need a coefficient?}\mnote{we don't need to conjugate here?} for $i=1,\ldots,m$, then $v_0,v_1,\ldots,v_m$ are a corresponding basis for the first block.
%Next, we will determine the basis vector of $W_2$ with the largest eigenvalue and apply the same procedure. But note that the calculation of such a vector is more difficult in general: if $m$ and $n$ have different parity, again the eigenspace of $n$ will have dimension one and any vector will do, but if the parity is the same, we cannot tell it apart so easily because the eigenspace has dimension two. In the latter case, we will use additional information: note that the representation of $x$ is\mnote{is it???}
\[
 \setlength{\arraycolsep}{0.3em}
 Rep(y)=\left(\begin{array}{ccccc|ccccc}
 0 &&&&& \\
 m & 0 &&&&& \\
 & \ddots & \ddots &&&&& \\
 & & 2 & 0 &&&&& \\
 & & & 1 & 0 &&&& \\
 \hline
 &&&&& 0 \\
 &&&&& n & 0 \\
 &&&&&& \ddots & \ddots \\
 &&&&&&& 2 & 0 \\
 &&&&&&&& 1 & 0 \\
 \end{array}\right)
\]  
\end{remark}

\begin{theorem}\label{TH: line fibration for Smn with m>n}
Let $X\in\pj^{m+n+1}$ be a variety isomorphic to $S_{m,n}$ with $m>n$ and let $v,w\in k^{m+n+2}$ such that $v$ is an eigenvector of $Rep(h)$ with largest eigenvalue and $w=Rep(y)\cdot v$. Then the function $X\rightarrow\pj^1$ defined by $w/v$ has as its fibers the lines of $X$.
\end{theorem}

\begin{proof}
We prove this for the scroll and the result will follow since the isomorphism will respect the construction. So we consider $S_{m,n}$ which is the image of $(s:t)\mapsto(1:s:s^2:\ldots:s^m:t:st:s^2t:\ldots:s^nt)\subset\pj^{m+n+1}$; denote the coordinates of $\pj^{m+n+1}$ as $x_{0,0},\ldots,x_{0,m},x_{1,0},\ldots,x_{1,n}$. By Remark \ref{REM: Images of h,x,y in two-block rep} we must have $v=(\lambda,0,\ldots,0)$ and $w=(0,m\lambda,0,\ldots,0)$. In terms of the parametrization, $w/v$ on $S_{m,n}$ is $m\lambda s/\lambda=ms$ and its fibres are clearly lines.
\end{proof}

It is worth noting that one can extend this procedure to calculate a basis of eigenvectors of $k^{m+n+2}$. One can determine the eigenvalues of $Rep(h)$ and read off $m$ and $n$. Then we just need an eigenvector for $m$ and $n$, but if $m,n$ have the same parity there is an ambiguity since the eigenspace for $n$ has dimension 2; it suffices to intersect it with the kernel of $Rep(x)$ in order to isolate the correct unidimensional eigenspace for $n$. By succesive application of $Rep(y)$ we complete the basis of eigenvectors, and the conversion from the canonical basis to the new basis produces a linear isomorphism of $\pj^{m+n+1}$ which restricts to an isomorphism between $X$ and $S_{m,n}$.
%We see here that the kernel of this matrix is generated precisely by the first vector of each block. Thus, the sought vector is in the intersection of the kernel of $x$ and the eigenspace of the eigenvalue $n$.

%Another possibility to solve this case is to intersect the space of matrices which have bracket with $y$ equal to zero, with the nilradical of the Lie algebra. This subspace has dimension one and any element will send $v_0$ to an eigenvector of $n$ in the second block, as needed.\mnote{too technical?}

\subsection{The case $m=n$}\label{SEC: Case m=n}

The idea is similar to the previous case but the details are somewhat different.

One can obtain representations of $2\fsl_2$ from representations of $\fsl_2$, or in other words $2\fsl_2$-modules from $\fsl_2$-modules. Indeed, in general, if $V_1$ is a $L_1$-module and $V_2$ is a $L_2$-module, then $V_1\otimes V_2$ is a $(L_1+L_2)$-module via the action $(l_1+l_2)\cdot(v_1\otimes v_2)=(l_1\cdot v_1)\otimes v_2+v_1\otimes(l_2\cdot v_2)$.

We saw in Section \ref{SEC: Reps of Lie algebras} that the irreducible representations of $\fsl_2$ are classified by their dimension. Likewise, every irreducible representation of $2\fsl_2$ is the tensor product of two irreducible representations of $\fsl_2$ and determined by the dimensions of the two parts. This is due to the characterization of a representation of a semisimple algebra being irreducible precisely when it has a highest weight (see \cite[Chapter 8]{deGraaf2000a} for definitions and properties), and the fact that the tensor product of two irreducible representations of $\fsl_2$ has a highest weight given by the highest eigenvalues of the matrices corresponding to the respective $h$ elements in their Chevalley bases (see Remark \ref{REM: Images of h,x,y in Sym^n}).

On the other hand, the representation of the model $S_{m,m}\subset\pj^{2m+1}$ is the product of the $2$-dimensional and the $(m+1)$-dimensional representations of $\fsl_2$ (see \cite[Section 3.4]{Oda1988a}, or calculate for particular values of $m$). In other words, we have two representations of $\fsl_2$ corresponding to the two summands of $2\fsl_2$. Both are given, in terms of matrices, by the Kronecker product (see \cite[Section 4.2]{HornJohnson1991a}) of the respective components (recall that the Kronecker product is not commutative); in particular, the image of $h$ by the tensor product representations is up to similarity one of the Kronecker products of
\[
 \setlength{\arraycolsep}{0.1em}
 \begin{pmatrix} m\\ &m-2\\ &&\ddots\\ &&&-m \end{pmatrix}
 \qquad\mbox{and}\qquad
 \begin{pmatrix} 1&0\\0&-1 \end{pmatrix}
\]
which are
\[
 \setlength{\arraycolsep}{0.1em}
 \left(\begin{array}{cccc|cccc}
 m &&&&& \\
 & m-2 &&&&& \\
 & & \ddots &&&&& \\
 & & & -m &&&& \\
 \hline
 &&&& m & \\
 &&&& & m-2 & \\
 &&&& & & \ddots & \\
 &&&& & & & -m \\
 \end{array}\right)
 \qquad\mbox{or}\qquad
 \left(\begin{array}{ccccccc}
 1 && \multicolumn{1}{|c}{} \\
 & -1 & \multicolumn{1}{|c}{} \\ \cline{1-4}
 && \multicolumn{1}{|c}{1} && \multicolumn{1}{|c}{} \\
 && \multicolumn{1}{|c}{} & -1 & \multicolumn{1}{|c}{} \\ \cline{3-4}
 &&&& \ddots \\ \cline{6-7}
 &&&&& \multicolumn{1}{|c}{1} \\
 &&&&& \multicolumn{1}{|c}{} & -1 \\
 \end{array}\right)
\]
In short, we have two copies of $\fsl_2$ acting on the underlying vector space, one of them decomposes into irreducibles as a sum of $m$ two-dimensional representations and the other as a sum of two $m$-dimensional representations.

Once we decompose $2\fsl_2$ into two copies of $\fsl_2$, how to discern the two tensor product representations? Note that it is only needed if $m>1$. We can decide this by considering the matrices for $h$. Note that the square of a matrix similar to the right hand side is the identity, which is not the case for the matrix on the left. So we can identify this case, for example, by picking one of the two matrices and checking if the degree of its minimal polynomial is 2. Once we have distinguished the two representations, we concentrate on the left hand side representation. Call it $Rep\colon\fsl_2\rightarrow\fgl(k^{2m+2})$. We obtain a result very similar to that of the case $m\neq n$.

\begin{theorem}\label{TH: line fibration for Smm}
Let $X\in\pj^{2m+1}$ be a variety isomorphic to $S_{m,m}$ and let $v,w\in k^{2m+2}$ such that $v$ is an eigenvector of $Rep(h)$ with eigenvalue $m$ and $w=Rep(y)\cdot v$. Then the function $X\rightarrow\pj^1$ defined by $w/v$ has as its fibers the lines of $X$.
\end{theorem}

\begin{proof}
It suffices to prove this for the scroll $S_{m,m}$. We consider $S_{m,m}$ which is the image of $(s:t)\mapsto(1:s:s^2:\ldots:s^m:t:st:s^2t:\ldots:s^mt)\subset\pj^{2m+1}$; denote the coordinates of $\pj^{2m+1}$ as $x_{0,0},\ldots,x_{0,m},x_{1,0},\ldots,x_{1,m}$. Then by Remark \ref{REM: Images of h,x,y in two-block rep} any such $v$ must be equal to $(\lambda_1,0,\ldots,0,\lambda_2,0,\ldots,0)$ with the possibly nonzero entries at the positions $x_{0,0},x_{1,0}$. Then the coordinates of $w$ are the result of a right-shift and multiplication by $m$. As a result, the map defined by $w/v$ on $X$ is just $(m\lambda_1s+m\lambda_2st)/(\lambda_1+\lambda_2t)=ms$. Clearly its fibres are lines.
\end{proof}

\section{Computational experiences}\label{SEC: Computational experiences}

We have tested our implementation in Magma V2.14-7 against random examples of trigonal curves over the field of rational numbers. The computer used is a 64 Bit, Dual AMD Opteron Processor 250 (2.4 GHZ) with 8 GB RAM. We have generated trigonal curves in the following two ways:

\begin{enumerate}
 \item Let $C:f(x,y,z)=0$ homogeneous with $\deg_y f=3$. Then the projection $(x:y:z)\mapsto(x:z)$ is a $3:1$ map to $\pj^1$. The genus of a curve of degree 3 in $y$ and degree $d$ in $x$ is $2(d-1)$ generically. The size of the coefficients is controlled directly.

 \item Let $C$ be defined by the affine equation $\Resultant_u(F,G)=0$ where
\[\begin{array}{l@{\ }l}
0 = x^3-a_1(u)x-a_2(u) & =:F, \\
0 = y-a_3(u)-a_4(u)x-a_5(u)x^2 & =:G
\end{array}\]
for some polynomials $a_1,\ldots,a_5$. This clearly gives a field extension of degree 3, thus there is a $3:1$ map from $C$ to the affine line. Examples show that the degree and coefficient size for a given genus are significantly larger than for the previous construction.
\end{enumerate}

It is important to remark that we need not compute either the genus or the hyperellipticity of the curve beforehand, since both things are detected by the algorithm: the genus is a byproduct of the computation of the canonical image, and the canonical image of a hyperelliptic curve is a rational normal curve, a situation detected by the Lie algebra computation.

For the same reason, we have only performed our timing tests on trigonal curves: when the intersection of the quadrics containing the canonical curve is unidimensional or isomorphic to the Veronese surface, this will be detected by inspecting the dimension of the whole Lie algebra or of the Levi subalgebra respectively.

These are the time results (in seconds) for samples of fifty random curves of degree 3 in one of the variables, for various values of total degree, genus and coefficient size. Note the effect of the coefficient size on the computing time.

\[\begin{array}{ccc|ccc}
\mathrm{genus} & \deg_x & \mathrm{bit\ height} & \mathrm{min} & \mathrm{avg} & \mathrm{max} \\ \hline
4 & 3 & 8 & 0.020 & 0.030 & 0.080 \\
4 & 3 & 200 & 0.140 & 2.606 & 36.240 \\
18 & 10 & 8 & 26.880 & 27.843 & 28.410 \\
30 & 16 & 2 & 1690.340 & 2099.167 & 2435.630 \\ % [2182.240, 2026.450, 2038.180, 2057.220, 2093.790, 2065.080, 2039.870, 2292.820, 2006.410, 2035.990, 2039.250, 2066.500, 2061.050, 2048.200, 2070.890, 2066.840, 2074.240, 1690.340, 2079.750, 2079.450, 2066.950, 2072.540, 2113.150, 2435.630, 2120.450, 2121.520, 2123.100, 2118.800, 2189.050, 2108.440, 2140.330, 2132.000, 2181.100, 2076.930, 2066.480, 2124.000, 2138.860, 2173.460, 2091.920, 2149.300, 2100.100, 2088.090, 2106.350, 2131.580, 2177.840]
30 & 16 & 8 & \multicolumn{3}{c}{\mbox{out of memory}} \\
\end{array}\]
% \[\begin{array}{cc|ccc}
% \deg_x & \mathrm{bit\ height} & \mathrm{genus} & \deg & \mathrm{seconds} \\ \hline
% 3 & 5 & 4 & 6 & 0.5-0.65 \\
% 3 & 50 & 4 & 6 & 2.09-2.27 \\
% 6 & 5 & 10 & 9 & 14-17 \\
% 6 & 50 & 10 & 9 & 54-61 \\
% 10 & 5 & 18 & 13 & 271-342 \\
% 10 & 50 & 18 & 13 & 1059-1193 \\
% 15 & 5 & 28 & 18 & 3477-5317 \\
% \end{array}\]

For the second method, we choose $a_1,\ldots,a_5$ randomly of degree $d$ and integer coefficients between $-e$ and $e$. These are the time results (in seconds) for samples of fifty random curves, for various values of $d,e$.

\[\begin{array}{c|ccc|ccc}
(d,e) & \mathrm{genus} & \deg & \mathrm{bit\ height} & \mathrm{min} & \mathrm{avg} & \mathrm{max} \\ \hline
(4,2) & 3-4 & 17-20 & 10-18 & 0.030 & 0.207 & 0.980 \\
(4,2000) & 4 & 20 & 85-95 & 0.630 & 0.847 & 1.490 \\
(5,2) & 3-6 & 20-25 & 16-21 & 0.050 & 1.939 & 2.890 \\
(5,200) & 6 & 25 & 76-86 & 6.680 & 8.975 & 11.370 \\
\end{array}\]

% \[\begin{array}{c|cccc}
% (d,e) & \mathrm{genus} & \deg & \mathrm{bit\ height} & \mathrm{seconds} \\ \hline
% (4,2) & 4 & 15-20 & 9-16 & 18-62 \\
% (4,10) & 4 & 20 & 26-35 & 87-191 \\
% (5,2) & 4-6 & 20-25 & 17-21 & 162-2353  \\
% (5,10) & 4-6 & 23-25 & 34-41 & 1334-7940 \\
% (6,2) & 6-7 & 25-30 & 20-24 & 2992-22650 \\
% \end{array}\]

Our Magma implementation can be obtained by contacting us directly.

\bibliographystyle{amsplain}

\end{document}